\RequirePackage{fix-cm}
\documentclass[smallcondensed]{svjour3}       
\smartqed  
\usepackage{amssymb}
\usepackage{bm}
\usepackage{graphicx}
\usepackage{graphics}
\usepackage{psfrag}
\usepackage{enumitem}
\usepackage{listings}
\usepackage{cite}
\usepackage{float}
\usepackage{arydshln}
\usepackage{amsmath}
\usepackage{amscd}
\usepackage{amsfonts}
\usepackage{float,verbatim}
\usepackage{latexsym}
\usepackage{color}
\usepackage{multirow}
\usepackage{booktabs}
\usepackage{lscape}
\usepackage{arydshln}
\usepackage{epstopdf}
\usepackage{cases}
\usepackage{lineno}

\newcommand{\p}{\partial}

\newcommand{\ds}{\displaystyle}

\begin{document}

\title{Numerical analysis for subdiffusion problem with non-positive memory
}

\titlerunning{Numerical analysis for subdiffusion problem}        

\author{Wenlin Qiu\textsuperscript{1}         \and
        Xiangcheng Zheng\textsuperscript{2}
}

\authorrunning{W. Qiu, X. Zheng} 

\institute{W. Qiu \\
              \email{wlqiu@sdu.edu.cn}  \\
           \at
              X. Zheng (Corresponding author) \\
           \email{xzheng@sdu.edu.cn} \\
           \at
               {1} School of Mathematics, Shandong University, Jinan, Shandong 250100, P. R. China \\
           \at
               {2} School of Mathematics, State Key Laboratory of Cryptography and Digital Economy Security, Shandong University, Jinan 250100, P. R. China
    }

\date{Received: date / Accepted: date}

\maketitle

\begin{abstract}
  This work considers the subdiffusion problem with non-positive memory, which not only arises from physical laws with memory, but could be transformed from sophisticated models such as subdiffusion or subdiffusive Fokker-Planck equation with variable exponent.  We apply the non-uniform L1 formula and interpolation quadrature to discretize the fractional derivative and the memory term, respectively, and then adopt the complementary discrete convolution kernel approach to prove the stability and first-order temporal accuracy of the scheme. The main difficulty in numerical analysis lies in the non-positivity of the kernel and its coupling with the complementary discrete convolution kernel (such that different model exponents are also coupled), and the results extend those in [Chen, Thom\'ee and Wahlbin, {\it Math. Comp.} 1992] to the subdiffusive case. Numerical experiments are performed to substantiate the theoretical results.

\keywords{subdiffusion \and non-positive memory \and variable step size \and finite element method \and error estimate}

 \subclass{45K05 \and 65M12 \and 65M60}
\end{abstract}

\section{Introduction}
\subsection{Model problem}
This work considers the following subdiffusion problem with memory, which could be derived from, e.g. the constitutive laws of the flow in  porous media with memory  \cite{Kra3,Mahata,Mahata1}
\begin{align}
  & _{}^{c}\p_t^{\sigma}u(\bm x,t) - \Delta u(\bm x,t) = (K * B u)(\bm x,t) + f(\bm x,t), ~~(\bm x,t) \in \Omega\times(0,T]; \label{model}  \\
  &\qquad u(\bm x,0)=u_0(\bm x),~\bm x\in \Omega; \quad u(\bm x,t) = 0,~(\bm x,t) \in \p \Omega\times[0,T]. \label{IBV}
\end{align}
Here $\Omega$ is a bounded domain in $\mathbb R^d$ ($1\leq d\leq 3$) with the boundary $\partial\Omega$, $\bm x := (x_1,\cdots,x_d)$, $T>0$, $\Delta$ denotes the Laplace operator, and $f$ and $u_0$ refer to the source term and the initial value, respectively.  The Caputo fractional derivative is defined via the convolution  (denoted by $*$) and the standard Gamma function $\Gamma(\cdot)$ \cite{Jin,JinLi,JinLi2}
\begin{align}
   _{}^{c}\p_t^{\sigma}  \varphi(t) & = \tilde{k}_{1-\sigma}(t) * \p_t\varphi(t), \quad 0< \sigma<1; ~~
    \tilde{k}_{\vartheta}(t)  = \frac{t^{\vartheta-1}}{\Gamma(\vartheta)}, \quad \vartheta >0, \label{op}
\end{align}
the
$   B:  = \mu_1 \Delta  + \mu_2 I
$ with $ |\mu_1| + |\mu_2| < \infty$, and the
  kernel satisfies
\begin{align}
   |K(t)| & \leq \kappa t^{-\beta},\quad t\in (0,T], \quad 0 < \beta < \min(1,2\sigma), \label{condi}
\end{align}
for some $\kappa>0$. It is worth mentioning that this condition only gives the bound of $K$ but does not impose other favorable properties such as the positive definiteness, a critical property in conventional numerical analysis of partial integrodifferential equations.

Apart from deriving  \eqref{model}--\eqref{IBV} from underlying mechanisms, the proposed model could also be transformed from other sophisticated models. Two typical scenarios are given as follows.

\underline{Subdiffusion equation with variable exponent} is applied in several fields such as the anomalous modeling in the heterogeneous media \cite{Sun1,Sun2}
\begin{equation}\label{VtFDE1}
   _{}^{c}\p_t^{\alpha(t)}u(\bm x,t)  - \Delta u (\bm x,t) = f(\bm x,t),
\end{equation}
where the Caputo fractional derivative $_{}^{c}\p_t^{\alpha(t)}$ with variable exponent $0< \alpha(t)<1$ is defined as
\cite{LorHar}
\begin{align}
& _{}^{c}\p_t^{\alpha(t)}  \varphi(t) = \tilde{k}_{1-\alpha(t)}(t) * \p_t\varphi(t), \quad t>0. \label{opr1}
\end{align}
 Due to the presence of the variable-exponent kernel, the operator \eqref{opr1} in problem \eqref{VtFDE1} lacks favorable properties as its constant-exponent analogue, which makes the analysis intricate. To resolve this issue, a perturbation method is proposed in \cite[Section 5]{Zheng} to transform problem \eqref{VtFDE1} into the following equivalent model with $\alpha_0=\alpha(0)$
\begin{align}\label{modelA}
  & _{}^{c}\p_t^{\alpha_0}u - \Delta u = \tilde{g}^{\prime} * (-u) + (f+\tilde{g} u_0),
\end{align}
where
\begin{align*}
   \tilde{g}(t) = \int_{0}^{t}\tilde k_{1-\alpha(z)}(t) \left[ -\alpha'(z)\ln t + \frac{\Gamma'(z)}{\Gamma(z)}(1-\alpha(z))\alpha'(z) \right] dz
\end{align*}
has the following estimates for $t\in (0,T]$ and for some constants $C_1, C_2>0$
\begin{align*}
   |\tilde{g}(t)| &\leq C_1 t^{1-\alpha_0} (1+|\ln t|) \rightarrow 0 \;\; \text{as} \;\; t\rightarrow 0^+, \\
   |\tilde{g}^{\prime}(t)| &\leq C_1 t^{-\alpha_0} (1+|\ln t|) \leq C_2 t^{-(\alpha_0+\varepsilon)}, \;\; 0<\varepsilon  <\min\{\alpha_0,1-\alpha_0\}.
\end{align*}
   Thus, model \eqref{modelA} is a special case of \eqref{model} with $\sigma=\alpha_0$, $\beta = \alpha_0+\varepsilon$, $\mu_1=0$ and $\mu_2=-1$, the latter of which is more tractable.

\underline{Subdiffusive Fokker-Planck equation with variable exponent} is recently proposed in \cite{ZheLiQiu} to resolve the initial singularity of the solutions without affecting the long-term power-law behavior
\begin{equation}\label{VtFDE2}
  \p_t u(\bm x,t)- \p_t^{1-\alpha(t)} \Delta u (\bm x,t)= f(\bm x,t).
\end{equation}
Here $\p_t^{1-\alpha(t)}$ with the variable exponent $0< \alpha(t)<1$ denotes the Riemann-Liouville fractional derivative operator defined as follows
\cite{LorHar}
\begin{align}
& \p_t^{1-\alpha(t)}  \varphi(t) = \p_t \big(\tilde{k}_{\alpha(t)}(t)*\varphi(t) \big).  \label{opr2}
\end{align}
 Note that in the case of variable exponent, the semigroup property of the fractional integral operator is not valid such that one could not transform \eqref{VtFDE2} to \eqref{VtFDE1}  as in the constant-exponent case. To circumvent the difficulties caused by the variable-exponent operator, one could apply the convolution method in \cite[Section 3]{Zheng} to transform \eqref{VtFDE2} as
\begin{align}\label{modelB}
  & _{}^{c}\p_t^{\alpha_0}u - \Delta u = g^{\prime} * \Delta u + (\tilde{k}_{1-\alpha_0}*f),
\end{align}
where
\begin{align*}
   g(t) = \int_{0}^{1} \frac{(tz)^{\alpha(tz)-\alpha(0)}}{\Gamma(1-\alpha(0))\Gamma(\alpha(tz))} (1-z)^{-\alpha(0)}z^{\alpha(0)-1} dz
\end{align*}
satisfies
$$  |g^{\prime}(t)| \leq C_3 (1+|\ln t|) \leq C_4 t^{-\varepsilon},~~0<\varepsilon <\alpha_0$$
 for $t\in (0,T]$ and for some constants $C_3,C_4>0$.
Thus, model \eqref{modelB} is also a special case of \eqref{model} with $\sigma=\alpha_0$, $\beta = \varepsilon$, $\mu_1=1$ and $\mu_2=0$.

From the above discussions, we find that the model (\ref{model})--(\ref{IBV}) covers a broad class of equations arising in several fields, which motivates the investigation of its numerical computation.

\subsection{Main contribution}
For the ordinary differential equation analogues of (\ref{model})--(\ref{IBV}), there exist sophisticated investigations, see e.g. \cite{CheChe,Che2,MaHua,Yi1,Yi2}. For the partial differential equations like model (\ref{model})--(\ref{IBV}), several works are devoted for the special case $\sigma=1$ with smooth kernel \cite{Lin,Pani,Xu1,Xu2} or weakly singular positive-type kernel \cite{Mcl,Mus1,Mus2,Wang}, and for the subdiffusive case $0<\sigma<1$ with smooth kernel \cite{He,Kumar,Mahata1} or weakly singular positive-type kernel \cite{Mahata,Zho}. For the special case $\sigma=1$, there are few works where the kernel only satisfies (\ref{condi}) (such that it may not be positive definite). For instance, the \cite{Chen} proposes and analyzes a first-order time-discretization scheme for model (\ref{model})--(\ref{IBV}) with $\sigma=1$, where a discrete weakly singular Gronwall inequality is proposed to account for the singularity of the memory kernel.

For the case $0<\sigma<1$, there exists a very recent work \cite{Zheng} which performs numerical analysis for a related Volterra integro-differential equation with non-positive kernel transformed from (\ref{VtFDE1}). In \cite{Zheng} the proved temporal convergence rate is $\frac{1}{2}+\frac{3}{2}\alpha_0$, which is, roughly speaking, only half-order convergence for small $\alpha_0$. Furthermore, the exponential factor method is adopted in the numerical analysis of \cite{Zheng} due to the specific structure of the model, which does not work for (\ref{model})--(\ref{IBV}) due to the existence of the fractional derivative.
To our best knowledge, the unified temporal accuracy  for numerical discretization of  (\ref{model})--(\ref{IBV}) is not available in the literature, which motivates the current study.

This work extends the work \cite{Chen} by considering the subdiffusive case $0 < \sigma < 1$, where the L1 method \cite{Stynes} and piecewise-constant interpolation method are used to discretize the fractional derivative and the memory term, respectively.
In this case, the solutions have the asymptotics $\partial_t^2 u \sim t^{\sigma-2}$ at $t=0$ that is much singular than $\partial_t^2 u \sim Ct^{-\beta}$ for the case $\sigma=1$ such that a graded temporal mesh may be required to get the first-order accuracy. This, together with the nonlocality of the fractional derivative, deteriorates the uniform time-step analysis framework in \cite{Chen}. Furthermore, as the regularity of the solutions depends on $\sigma$, the two parameters $\sigma$ and $\beta$ will be coupled in time discretization analysis, which makes the derivations more complicated. To account for these difficulties, we adopt the complementary discrete convolution kernel approach \cite{Liao} and develop the convolution-kernel-based norm to establish the stability and error estimate of the scheme.

The rest of the work is organized as follows: In Sections \ref{sec2}--\ref{sec3}, we construct and analyze  semi- and fully-discrete schemes of the model (\ref{model})--(\ref{IBV}), respectively, and we perform numerical experiments to validate theoretical results in Section \ref{sec4}. The conclusion of this work is given in Section \ref{sec5}. In subsequent analysis, we define the norm of the $L^2(\Omega)$ space by
$$\|g\|:=\Big(\int_\Omega g^2(\bm x)d\bm x\Big)^{1/2}$$
 and define $H^m(\Omega)$ and $H^m_0(\Omega)$ with $m\in\mathbb N^+$ in a standard manner with standard norms \cite{AdaFou}. Furthermore, we use $Q$ to denote a generic positive constant that may assume different values at different occurrences.

\section{Time-discrete scheme} \label{sec2}
We consider temporal discretization of  \eqref{VtFDE1}--\eqref{VtFDE2} based on the following regularity conditions, which are motivated from analysis of similar models \cite{Kra2,Mahata}
\begin{align}\label{regu}
   & \|\p_t u(\cdot,t)\| + \|\p_t \Delta u(\cdot,t)\| \leq Q t^{\sigma-1}, \quad \|\p_{t}^2 u(\cdot,t)\| \leq  Q t^{\sigma-2}, \quad t\in (0,T].
\end{align}

\subsection{Numerical scheme}
Let $N$ be a positive integer and $\tau_n=t_n-t_{n-1}$ for $1\leq n \leq N$ be the nonuniform temporal step size with
\begin{align*}
  t_n=T\Big(\frac{n}{N}\Big)^{\gamma}, \quad \gamma\geq 1, \quad 0\leq n \leq N.
\end{align*}
Let $u^n:=u(\bm x,t_n)$. We consider the model \eqref{model} at $t=t_n$ for $1\leq n \leq N$
\begin{equation}\label{qwl03}
   _{}^{c}\p_t^{\sigma}u(\bm x,t_n) - \Delta u(\bm x,t_n) = (K * B u)(\bm x,t_n) + f(\bm x,t_n).
\end{equation}
We discretize the Caputo fractional derivative $^c\p_t^{\sigma} u$ by the following nonuniform $L_1$ formula \cite{Liao,Stynes}
\begin{equation}\label{qwl04}
  \begin{split}
      & D_N^{\sigma} u^n = \sum\limits_{k=1}^{n} \int_{t_{k-1}}^{t_k} \tilde{k}_{1-\sigma}(t_n-s) \delta_{\tau}u^k/\tau_k ds = \sum\limits_{k=1}^{n} a^{(n)}_{n-k} \delta_{\tau}u^k,
  \end{split}
\end{equation}
where $\delta_{\tau}u^k=u^{k}-u^{k-1}$ for $k\geq 1$ and the nonincreasing sequence $\left\{a^{(n)}_{n-k}\right\}_{k=1}^{n}$ is defined as follows
\begin{equation}\label{qwl05}
  \begin{split}
      a^{(n)}_{n-k} = \int_{t_{k-1}}^{t_k} \frac{\tilde{k}_{1-\sigma}(t_n-s)}{\tau_k}ds = \frac{\tilde{k}_{2-\sigma}(t_n-t_{k-1})-\tilde{k}_{2-\sigma}(t_n-t_{k})}{\tau_k}.
  \end{split}
\end{equation}
Based on the regularity assumptions given in \eqref{regu}, we have the $L_1$ approximation error of \eqref{qwl04}  (see \cite{Stynes} for more details)
\begin{equation}\label{qwl06}
  \begin{split}
      \left\|(R_1)^n \right\| = \left\| D_N^{\sigma} u^n - \; ^c\p_t^{\sigma} u(\bm x,t_n) \right\| = O(n^{-\min(2-\sigma,\gamma\sigma)}), \quad n\geq 1.
  \end{split}
\end{equation}
 We apply the interpolation quadrature rule to approximate  the convolution term on the right-hand side of \eqref{qwl03}
\begin{equation}\label{qwl07}
\begin{split}
     (K * B u) (\bm x,t_n) & = \sum\limits_{j=1}^{n} \int_{t_{j-1}}^{t_j} K(s) B u (\bm x,t_n-s) ds \approx \sum\limits_{j=1}^{n} w_j B u^{n-j}
\end{split}
\end{equation}
with the weight
\begin{equation}\label{qwl08}
\begin{split}
     w_j := \int_{t_{j-1}}^{t_j} K(s)ds \leq \kappa \mu_{\beta,j}, \quad  1\leq j \leq n, \quad  0<\beta< \min(1,2\sigma),
\end{split}
\end{equation}
where
\begin{align}
    \mu_{\beta,j} := \int_{t_{j-1}}^{t_j} s^{-\beta} ds = \frac{t_j^{1-\beta} - t_{j-1}^{1-\beta} }{1-\beta}. \label{qwl0801}
\end{align}
The quadrature error of the numerical approximation \eqref{qwl07} could be expressed as
\begin{equation}\label{qwl09}
  \begin{split}
      (R_2)^n = \sum\limits_{j=1}^{n}\int_{t_{j-1}}^{t_j}K(s)   \int_{s}^{t_j} \p_t\Delta u(\bm x,t_n-\theta)d\theta  ds.
  \end{split}
\end{equation}
We then substitute \eqref{qwl04} and \eqref{qwl07} into \eqref{qwl03} to get
\begin{equation}\label{qwl010}
  \begin{split}
      D_N^{\sigma} u^n - \Delta u^n = \sum\limits_{j=1}^{n} w_j B u^{n-j} + f^n + R^n, \quad n = 1,2,\cdots,N
  \end{split}
\end{equation}
with $f^n=f(\bm x,t_n)$ and $R^n=(R_1)^n+(R_2)^n$. Omit the truncation error $R^n$ and replace $u^n$ with its approximation $U^n$ to get the following time semi-discrete scheme
\begin{align}
    & D_N^{\sigma} U^n - \Delta U^n = \sum\limits_{j=1}^{n} w_j B U^{n-j} + f^n, \quad n\geq 1, \label{qwl011} \\
    & U^0 = u_0. \label{qwl012}
\end{align}

\subsection{Stability analysis}
We deduce the stability of the time semi-discrete scheme  \eqref{qwl011}--\eqref{qwl012}. To facilitate analysis, we introduce the following complementary discrete convolution kernel \cite{Liao,Liao1}
\begin{equation}\label{doc}
   P_{n-k}^{(n)} = \frac{1}{a_0^{(k)}}
  \begin{cases}
    1, & k=n, \\
    \sum\limits_{j=k+1}^{n}\left( a_{j-k-1}^{(j)} - a_{j-k}^{(j)}  \right)P_{n-j}^{(n)}, & 1,2,\cdots,n-1.
  \end{cases}
\end{equation}

\begin{lemma}\cite{Liao,Liao1} \label{lem01}
   The complementary discrete convolution kernel $ P_{n-k}^{(n)}$ satisfies the following properties: \par\vskip 1mm
   (i)  $0< P_{n-k}^{(n)} \leq \Gamma(2-\sigma) \tau_k^{\sigma}$ for $1\leq k \leq n$; \par \vskip 1mm
   (ii) $\sum_{j=k}^{n}P_{n-k}^{(n)}a_{j-k}^{(j)}=1$ for $1\leq k \leq n$; \par \vskip 1mm
   (iii) For $q=0,1$ it holds
   \begin{align*}
      \sum_{k=1}^{n}  P_{n-k}^{(n)} \tilde{k}_{1+q\sigma - \sigma}(t_k) \leq \tilde{k}_{1+q\sigma}(t_n), \quad 1\leq n \leq N.
   \end{align*}
\end{lemma}

Then we proceed to analyze the global consistency error. The result of the following lemma could be derived from \cite[Lemma 3.3]{Liao}. However, for the case of graded meshes, we provide a simplified proof of this result.

\begin{lemma} \label{lem02}
  Under regularity assumption \eqref{regu},  we have
   \begin{align*}
       \sum_{n=1}^{m}  P_{m-n}^{(m)} \left\|(R_1)^n\right\| \leq Q(T)N^{-\min(2-\sigma,\gamma\sigma)}, \quad 1\leq n \leq  m \leq N.
   \end{align*}
\end{lemma}
\begin{proof} Let $\tilde{q}=\min(2-\sigma,\gamma\sigma)$ and we use \eqref{qwl06} to get
\begin{equation*}
 \begin{split}
    \sum\limits_{n=1}^{m}  P_{m-n}^{(m)} \left\|(R_1)^n\right\|  & \leq Q \sum_{n=1}^{m}  P_{m-n}^{(m)} n^{-\tilde{q}} \\
    & = Q \sum\limits_{n=1}^{m}  P_{m-n}^{(m)} \left[ T \left(\frac{n}{N}\right)^{\gamma} \right]^{-\tilde{q}/\gamma} N^{-\tilde{q}} T^{\tilde{q}/\gamma} \\
    & = QN^{-\min(2-\sigma,\gamma\sigma)} T^{\tilde{q}/\gamma}  \sum\limits_{n=1}^{m}  P_{m-n}^{(m)} t_n^{-\min\left(\frac{2-\sigma}{\gamma},\sigma\right)}.
 \end{split}
\end{equation*}
We combine this with
\begin{equation*}
   \sum\limits_{n=1}^{m}  P_{m-n}^{(m)} t_n^{-\min\left ( \frac{2-\sigma}{\gamma},\sigma\right)} =
   \begin{cases}
     \sum\limits_{n=1}^{m}  P_{m-n}^{(m)} t_n^{-\frac{2-\sigma}{\gamma}} = \sum\limits_{n=1}^{m}  P_{m-n}^{(m)} t_n^{\sigma-\frac{2-\sigma}{\gamma}} t_n^{-\sigma}, & \frac{2-\sigma}{\gamma}< \sigma, \\
     \sum\limits_{n=1}^{m}  P_{m-n}^{(m)} t_n^{-\sigma}, & \frac{2-\sigma}{\gamma} \geq  \sigma
   \end{cases}
\end{equation*}
and
\begin{align*}
      \sum\limits_{n=1}^{m}  P_{m-n}^{(m)} t_n^{-\sigma} &= \Gamma(1-\sigma) \sum\limits_{n=1}^{m}  P_{m-n}^{(m)} \tilde{k}_{1-\sigma}(t_n) \\
      & \leq \Gamma(1-\sigma) \tilde{k}_{1}(t_m) = \Gamma(1-\sigma)
\end{align*}
to complete the proof of the lemma.
\end{proof}

Next, we introduce a key discrete Gr\"{o}nwall lemma as follows.
\begin{lemma} \cite[Lemma 6]{Chen}  \label{lem03}
   Let $\mu_{\beta,j}$ be defined by \eqref{qwl0801}. Suppose that $z_n\geq 0$ and satisfies
   \begin{align*}
      z_n \leq y_n + Q_0 \sum_{j=0}^{n-1} \mu_{\beta,n-j} z_j, \quad n\geq 0,
   \end{align*}
in which $y_n\geq 0$ and $Q_0\geq 0$. Then for any $0<T<\infty$, it holds that
   \begin{align*}
      z_n \leq y_n + Q(T) \sum_{j=0}^{n-1} \mu_{\beta,n-j} y_j, \quad n\leq N.
   \end{align*}
\end{lemma}

We then present the stability result of the scheme \eqref{qwl011}--\eqref{qwl012}.

\begin{theorem}\label{thm4.1}
  Let $U^m$ be the solution of the scheme \eqref{qwl011}--\eqref{qwl012}. The following stability estimates hold
  \begin{align}
      \|U^m\|^2 \leq Q(T) \left( \|U^0\|^2 + N^{-\sigma} \|\nabla U^{0}\|^2 +  \max\limits_{1\leq n \leq m}\|f^n\|^2 \right), \quad 1\leq m \leq N.  \nonumber
    \end{align}
\end{theorem}
\begin{proof}
    We first take the inner product of \eqref{qwl011} with $2U^n$ to obtain
 \begin{align}
     2 \left( D_{N}^{\sigma}U^n, U^n \right)
     & + 2  \|\nabla U^n\|^2 = 2 \sum\limits_{j=1}^{n} w_j B(U^{n-j}, U^n) + 2 (f^n,U^n). \label{zxc05}
\end{align}
We then multiply \eqref{zxc05} by the factor $P^{(m)}_{m-n}$ and sum the resulting equation from $n=1$ to $n=m$ to get
 \begin{align}
     2 \sum_{n=1}^{m} &  P^{(m)}_{m-n}\left( D_{N}^{\sigma}U^n, U^n \right)
      + 2 \sum_{n=1}^{m} P^{(m)}_{m-n} \|\nabla U^n\|^2\nonumber\\
      & = 2 \sum_{n=1}^{m} P^{(m)}_{m-n} \sum\limits_{j=1}^{n} w_j B(U^{n-j}, U^n) + 2 \sum_{n=1}^{m} P^{(m)}_{m-n} (f^n,U^n)  \label{ww01} \\
      & =: \Theta_1 + \Theta_2. \nonumber
\end{align}
Next, we will analyze each term of \eqref{ww01}. First, from \cite[Theorem 2.1]{Liao} we have
 \begin{align*}
     2  \left( D_{N}^{\sigma}U^n, U^n \right)
     \geq  \sum_{k=1}^{n}a^{(n)}_{n-k} \delta_{\tau} (\|U^k\|^2).
\end{align*}
Thus, we exchange of the order of summation and apply Lemma \ref{lem01}(ii) to get
 \begin{align}
     \sum_{n=1}^{m} P^{(m)}_{m-n}  \sum_{k=1}^{n}a^{(n)}_{n-k} \delta_{\tau} (\|U^k\|^2) & =  \sum_{k=1}^{m} \delta_{\tau}(\|U^k\|^2)  \sum_{j=k}^{m} P^{(m)}_{m-n} a^{(n)}_{n-k} \nonumber \\
     & = \sum_{k=1}^{m} \delta_{\tau}(\|U^k\|^2) = \|U^m\|^2 - \|U^0\|^2. \label{ww02}
\end{align}
Second, we discuss the first term of the right-hand side of \eqref{ww01}. With \eqref{condi}, we employ Poincar\'{e} and Young inequalities to get
\begin{align*}
   |B(U^{n-j}, U^n)| \leq Q \|\nabla U^{n-j}\| \|\nabla U^{n}\| \leq  Q \|\nabla U^{n-j}\|^2 + \varepsilon^{*}\|\nabla U^{n}\|^2.
\end{align*}
Therefore, we have
 \begin{align}
    |\Theta_1| & \leq
     \varepsilon^{*} \sum_{n=1}^{m} P^{(m)}_{m-n} \sum\limits_{j=1}^{n} |w_j|  \|\nabla U^{n}\|^2  + Q\sum_{n=1}^{m} P^{(m)}_{m-n} \sum\limits_{j=1}^{n} |w_j| \|\nabla U^{n-j}\|^2  \nonumber \\
     & \leq \varepsilon^{*} \left(\sum\limits_{j=1}^{m} |w_j| \right)  \sum_{n=1}^{m} P^{(m)}_{m-n}  \|\nabla U^{n}\|^2+ Q\sum_{n=1}^{m} P^{(m)}_{m-n} \sum\limits_{j=1}^{n} |w_j| \|\nabla U^{n-j}\|^2 \nonumber  \\
     & =: \Phi_1 + \Phi_2. \label{ww03}
\end{align}
Then we note \eqref{qwl08} that
\begin{align}
     \sum\limits_{j=1}^{m} |w_j| & \leq  \kappa\sum\limits_{j=1}^{m} \mu_{\beta, j} \leq \kappa \int_{0}^{t_m} s^{-\beta} ds \leq \frac{\kappa}{1-\beta} t_m^{1-\beta} \leq \frac{\kappa T^{1-\beta}}{1-\beta}, \label{www}
\end{align}
which yields
\begin{align}
     \Phi_1 & \leq  \varepsilon^{*} \frac{\kappa T^{1-\beta}}{1-\beta} \sum_{n=1}^{m} P^{(m)}_{m-n}  \|\nabla U^{n}\|^2. \label{ww04}
\end{align}
To analyze $\Phi_2$, we interchange the order of summation to get
\begin{equation}\label{ww05}
    \begin{split}
        \Phi_2 & = Q \sum\limits_{j=1}^{m}   |w_j|  \sum_{n=j}^{m} P^{(m)}_{m-n}  \|\nabla U^{n-j}\|^2  =  Q \sum\limits_{j=1}^{m}   |w_j|  \sum_{n=j+1 }^{m} P^{(m)}_{m-n}  \|\nabla U^{n-j}\|^2  \\
        & + Q \sum\limits_{j=1}^{m}   |w_j|   P^{(m)}_{m-j}  \|\nabla U^{0}\|^2   =: \Phi_{21} +  \Phi_{22}.
    \end{split}
\end{equation}
 We use the index substitution $\ell = n-j$ and \eqref{qwl08} to bound $\Phi_{21}$ as
\begin{align}
     \Phi_{21}  & = Q \sum\limits_{j=1}^{m}   |w_j|  \sum_{\ell = 1}^{m-j} P^{(m-j)}_{m-j-\ell}  \|\nabla U^{\ell}\|^2   = Q \sum\limits_{n=0}^{m-1}   |w_{m-n}|  \sum_{\ell = 1}^{n} P^{(n)}_{n-\ell}  \|\nabla U^{\ell}\|^2 \nonumber \\
     & \leq Q \sum\limits_{n=0}^{m-1}   \mu_{\beta,m-n}  \sum_{\ell = 1}^{n} P^{(n)}_{n-\ell}  \|\nabla U^{\ell}\|^2.  \label{ww06}
\end{align}
Then for $\Phi_{22}$, we apply Lemma \ref{lem01}(i) and \eqref{www} to obtain
\begin{align}
   \Phi_{22}  \leq Q \Gamma(2-\sigma)\tau_m^{\sigma} \sum\limits_{j=1}^{m}   |w_j| \|\nabla U^{0}\|^2 \leq Q(T) N^{-\sigma} \|\nabla U^{0}\|^2. \label{ww07}
\end{align}
We invoke \eqref{ww04}--\eqref{ww07} into \eqref{ww03} to obtain
 \begin{align}
    \Theta_1 & \leq \varepsilon^{*}\frac{\kappa T^{1-\beta}}{1-\beta} \sum_{n=1}^{m} P^{(m)}_{m-n}  \|\nabla U^{n}\|^2  \nonumber \\
    & + Q \sum\limits_{n=0}^{m-1}  \mu_{\beta,m-n} \sum_{\ell = 1}^{n} P^{(n)}_{n-\ell}  \|\nabla U^{\ell}\|^2  + Q(T) N^{-\sigma} \|\nabla U^{0}\|^2.  \label{ww08}
\end{align}
We then use Poincar\'{e} and Young inequalities to bound $ \Theta_2$ as
\begin{align}
  | \Theta_2| & \leq Q \sum_{n=1}^{m} P^{(m)}_{m-n} \|f^n\| \|\nabla U^n\|   \nonumber\\
    & \leq \frac{1}{2} \sum_{n=1}^{m} P^{(m)}_{m-n}\|\nabla U^n\|^2 + Q \sum_{n=1}^{m} P^{(m)}_{m-n}\|f^n\|^2 \label{ww010} \\
    & \leq \frac{1}{2} \sum_{n=1}^{m} P^{(m)}_{m-n}\|\nabla U^n\|^2 + Q \max\limits_{1\leq n \leq m}\|f^n\|^2 \sum_{n=1}^{m} P^{(m)}_{m-n} \tilde{k}_{1}(t_n). \nonumber
\end{align}
 We now insert \eqref{ww02}, \eqref{ww08} and (\ref{ww010}) into \eqref{ww01}, set $\varepsilon^{*}=(\frac{\kappa T^{1-\beta}}{1-\beta})^{-1} = \frac{(1-\beta)T^{\beta-1}}{\kappa}$ to eliminate $\sum_{n=1}^{m} P^{(m)}_{m-n} \|\nabla U^n\|^2$ from the left-hand side of \eqref{ww01}, and utilize Cauchy-Schwarz inequality to get
 \begin{align}
    \|U^m\|^2 & +\frac{1}{2} \sum_{n=1}^{m} P^{(m)}_{m-n} \|\nabla U^n\|^2   \leq  \|U^0\|^2 \nonumber \\
    &\quad + Q(T) N^{-\sigma} \|\nabla U^{0}\|^2 +  Q \max\limits_{1\leq n \leq m}\|f^n\|^2 \sum_{n=1}^{m} P^{(m)}_{m-n} \tilde{k}_{1}(t_n)\nonumber \\
    & \quad\quad + Q \sum\limits_{n=0}^{m-1}  \mu_{\beta,m-n} \sum_{\ell = 1}^{n} P^{(n)}_{n-\ell}  \|\nabla U^{\ell}\|^2.
        \label{ww09}
\end{align}
Define
 \begin{align*}
   z_n : = \|U^n\|^2 + \frac{1}{2} \sum_{\ell = 1}^{n} P^{(n)}_{n-\ell}  \|\nabla U^{\ell}\|^2, \quad 1\leq n \leq m.
\end{align*}
 and utilize Lemma \ref{lem01}(iii) to reformulate (\ref{ww09}) as
 \begin{align}
    z_m   & \leq  \|U^0\|^2 + Q(T) N^{-\sigma} \|\nabla U^{0}\|^2 \nonumber \\
    & +  Q \max\limits_{1\leq n \leq m}\|f^n\|^2  + Q \sum\limits_{n=0}^{m-1}  \mu_{\beta,m-n} z_n.  \label{ww011}
\end{align}
Then we apply Lemma \ref{lem03} to get
\begin{align*}
   z_m \leq Q(T) \left( \|U^0\|^2 + N^{-\sigma} \|\nabla U^{0}\|^2 +  \max\limits_{1\leq n \leq m}\|f^n\|^2 \right), \quad 1\leq m \leq N,
\end{align*}
which finishes the proof of the theorem.
\end{proof}

\subsection{Convergence estimate}
Denote
\begin{align*}
  \rho^n = u^n - U^n, \quad n=0,1,2,\cdots,N.
\end{align*}
We subtract \eqref{qwl011} from \eqref{qwl010} to arrive at the following error equations
\begin{align}
    & D_N^{\sigma} \rho^n - \Delta \rho^n = \sum\limits_{j=1}^{n} w_j B \rho^{n-j} + R^n,   \quad n\geq 1, \label{yq03} \\
    & \rho^0  = 0. \label{yq04}
\end{align}
Based on Theorem \ref{thm4.1}, we derive the following convergence result.

\begin{theorem}\label{thm4.2}
   Let $u^m$ and $U^m$ be the solution of \eqref{qwl010} and \eqref{qwl011}, respectively. Then for $\gamma\geq \frac{1}{\sigma}$, we have
    \begin{align*}
       \|u^m-U^m\| \leq Q(T) N^{-1}, \quad 1\leq m \leq N.
    \end{align*}
\end{theorem}

\begin{proof}
    Based on \eqref{yq03}--\eqref{yq04}, a similar estimate as \eqref{ww09} leads to
 \begin{align}
    \|\rho^m\|^2 + \sum_{n=1}^{m} P^{(m)}_{m-n} \|\nabla \rho^n\|^2  & \leq  2 \sum_{n=1}^{m} P^{(m)}_{m-n} \|R^n\| \|\rho^n\| \nonumber \\
    & \quad + Q \sum\limits_{n=0}^{m-1}  \mu_{\beta,m-n} \sum_{\ell = 1}^{n} P^{(n)}_{n-\ell}  \|\nabla \rho^{\ell}\|^2.
        \label{ww012}
\end{align}
Denote
\begin{align*}
   \|\rho^n \|_A :=  \sqrt{\|\rho^n\|^2 + \sum_{\ell = 1}^{n} P^{(n)}_{n-\ell}  \|\nabla \rho^{\ell}\|^2}, \quad 1\leq n \leq m \leq N.
\end{align*}
Hence we rewrite \eqref{ww012} as
 \begin{align}
    \|\rho^m \|_A^2  & \leq  2 \sum_{n=1}^{m} P^{(m)}_{m-n} \|R^n\| \|\rho^n \|  + Q \sum\limits_{n=0}^{m-1}  \mu_{\beta,m-n} \|\rho^n \|_A^2,
        \label{ww013}
\end{align}
which, together with Lemma \ref{lem03}, implies that
 \begin{align}
    \|\rho^{m} \|_A^2  & \leq  2 \sum_{n=1}^{m} P^{(m)}_{m-n} \|R^n\|\|\rho^n \|    + Q(T) \sum\limits_{n=0}^{\bar{m}-1}  \mu_{\beta,m-n}  \sum_{j=1}^{n} P^{(n)}_{n-j} \|\rho^j \| \|R^j\| \nonumber \\
    &  \leq  2 \sum_{n=1}^{m} P^{(m)}_{m-n}  (\|(R_1)^n\|+\|(R_2)^n\|)\|\rho^n \|_A  \nonumber \\
     & + Q(T) \sum\limits_{n=0}^{m-1}  \mu_{\beta,m-n} \sum_{j=1}^{n} P^{(n)}_{n-j} \|\rho^j \|_A  (\|(R_1)^j\|+\|(R_2)^j\|).
    \label{ww015}
\end{align}
Next, we select $\bar{m}$ so that $\|\rho^{\bar{m}} \|_A = \max\limits_{1\leq n \leq m} \|\rho^n \|_A$. Then \eqref{ww015} yields
 \begin{align}
    \|\rho^{\bar{m}} \|_A  &  \leq  2 \sum_{n=1}^{\bar{m}} P^{(\bar{m})}_{\bar{m}-n}  (\|(R_1)^n\|+\|(R_2)^n\|) \nonumber \\
     & + Q(T) \sum\limits_{n=0}^{\bar{m}-1}  \mu_{\beta,\bar{m}-n} \sum_{j=1}^{n} P^{(n)}_{n-j}  (\|(R_1)^j\|+\|(R_2)^j\|).
        \label{ww014}
\end{align}
To bound the right-hand side of \eqref{ww014}, we apply Lemma \ref{lem02} to obtain
\begin{align}
    \sum_{j=1}^{n} P^{(n)}_{n-j} \|(R_1)^j\| \leq Q(T) N^{-\min(2-\sigma,\gamma\sigma)}, \quad n\geq 1. \label{ww016}
\end{align}
Then we remain to estimate $\|(R_2)^j\|$. Based on the regularity assumptions \eqref{regu}, we employ \eqref{qwl08} and \eqref{qwl09} to get
\begin{align*}
  \|(R_2)^j\| & \leq Q  \sum\limits_{\ell =1}^{j} \int_{t_{\ell-1}}^{t_{\ell}}  s^{-\beta} \int_{s}^{t_{\ell}}  (t_j-\theta)^{\sigma-1} d\theta ds \\
  & = Q  \sum\limits_{\ell =1}^{j-1} \int_{t_{\ell-1}}^{t_{\ell}}  s^{-\beta} \int_{s}^{t_{\ell}}  (t_j-\theta)^{\sigma-1} d\theta ds \\
  & + Q \int_{t_{j-1}}^{t_{j}}  s^{-\beta} \int_{s}^{t_{j}}  (t_j-\theta)^{\sigma-1} d\theta ds  =: (R_{21})^j + (R_{22})^j,
\end{align*}
where we have $(R_{21})^1=0$ and
\begin{align*}
  (R_{21})^j & \leq Q  \sum\limits_{\ell =1}^{j-1} \int_{t_{\ell-1}}^{t_{\ell}}  s^{-\beta}  (t_j-s)^{\sigma-1}  \tau_{\ell} ds \\
  &  \leq Q N^{-1} \int_{0}^{t_{j-1}} s^{-\beta}  (t_j-s)^{\sigma-1}  ds \\
  &  \leq Q N^{-1} \int_{0}^{t_{j}} s^{-\beta}  (t_j-s)^{\sigma-1}  ds \leq QN^{-1} t_j^{\sigma-\beta}, \quad j\geq 2.
\end{align*}
Then we further estimate $(R_{22})^j$ for $j\geq 1$. First, for $j=1$ we get
\begin{align*}
  (R_{22})^1 & \leq Q  \int_{0}^{t_{1}}  s^{-\beta}  (t_1-s)^{\sigma}  ds \\
  & \leq Q t_{1}^{\sigma} t_1^{1-\beta} = Q t_{1} t_1^{\sigma-\beta} \leq QN^{-1} t_1^{\sigma-\beta},
\end{align*}
and for $j\geq 2$ we have
\begin{align*}
  (R_{22})^j & \leq Q  \int_{t_{j-1}}^{t_{j}}  s^{-\beta}  (t_j-s)^{\sigma}  ds \leq Q \tau_j^{\sigma+1} t_j^{-\beta} \\
  & \leq Q \tau_j \left(\frac{\tau_j}{t_j}\right)^{\sigma} t_j^{\sigma-\beta} \leq QN^{-1} t_j^{\sigma-\beta}.
\end{align*}
Thus, the above two formulas give
\begin{align*}
  (R_{22})^j & \leq QN^{-1} t_j^{\sigma-\beta}, \quad j\geq 1.
\end{align*}
We then utilize the above estimates to get
\begin{align}
    \sum_{j=1}^{n} P^{(n)}_{n-j} \|(R_2)^j\| \leq Q(T) N^{-1}\sum_{j=1}^{n} P^{(n)}_{n-j} t_j^{\sigma-\beta}, \quad n\geq 1. \label{ouhe}
\end{align}
As $0<\beta< \min(2\sigma,1)$ and $0<\sigma<1$, we have
\begin{equation*}
   t_j^{\sigma-\beta} =
   \begin{cases}
     t_j^{2\sigma-\beta}t_j^{-\sigma} \leq T^{2\sigma-\beta} t_j^{-\sigma},  &  0<\sigma \leq 0.5, \;\; 0<\beta < 2\sigma, \\
     t_j^{0.5+\sigma-\beta} t_j^{\sigma-0.5} t_j^{-\sigma} \leq T^{2\sigma-\beta}t_j^{-\sigma}, & 0.5<\sigma <1, \;\; 0<\beta < 1.
   \end{cases}
\end{equation*}
We combine the above two formulas and Lemma \ref{lem01}(iii) to obtain
\begin{align}
    \sum_{j=1}^{n} P^{(n)}_{n-j} \|(R_2)^j\| \leq Q(T) N^{-1} \left(\sum_{j=1}^{n} P^{(n)}_{n-j} t_j^{-\sigma} \right) \leq Q(T) N^{-1}.  \label{ww017}
\end{align}
We invoke \eqref{ww016} and \eqref{ww017} into \eqref{ww014} to complete the proof.
\end{proof}

\section{Fully discrete Galerkin scheme}\label{sec3}

We shall formulate and analyze a fully discrete Galerkin scheme for the temporal semi-discrete scheme \eqref{qwl011}--\eqref{qwl012}.
Define a quasi-uniform partition of $\Omega$ with the mesh diameter $h$. Let $S_h$ be the space of continuous piecewise linear functions on $\Omega$ with respect to the partition. Define the Ritz projector $I_h: H^1_0(\Omega)\rightarrow S_h$ via
\begin{align}
     \big(\nabla (I_h \omega - \omega), \nabla \hat{\chi} \big) = 0 \quad \text{for} \;\; \hat{\chi} \in S_h \label{yq05}
\end{align}
with the following approximation property \cite{Tho}
\begin{align}
   \left\|\p_t^{\ell}(\omega - I_h \omega)\right\|\leq Q h^2 \left\|\p_t^{\ell}\omega(t)\right\|_{H^2(\Omega)}, \quad \ell=0,1. \label{yq06}
\end{align}

In view of the temporal semi-discrete scheme \eqref{qwl011}--\eqref{qwl012}, the fully discrete Galerkin scheme  reads: find $U^m_h\in S_h$ satisfying that for $1 \le n \le m \le N$:

\begin{align}
   \left(D_N^{\sigma} U^n_h, \hat{\chi} \right) + \left(\nabla U_h^{n}, \nabla\hat{\chi} \right)  =  \sum_{j=1}^{n} w_{j} B\left( U_h^{n-j}, \hat{\chi} \right)  +  \left(f^{n}, \hat{\chi} \right), \quad\forall \hat{\chi} \in S_h  \label{lkx01}
\end{align}
with the suitable approximation $U_h^0\approx u_0$. For the convenience of analysis, we denote
 \begin{align*}
     & u(t_n)- U_h^n =\zeta^n - \eta^n, \quad 0\leq n \leq m \le N,
 \end{align*}
where
\begin{align*}
    & \zeta^n = I_h u(t_n) - U_h^n \in S_h, \quad \eta^n = I_h u(t_n) - u(t_n).
\end{align*}
Based on \eqref{qwl010} and \eqref{lkx01}, we apply \eqref{yq05} to get
\begin{align}
   \left(D_N^{\sigma}\zeta^n, \hat{\chi} \right) + \left(\nabla \zeta^{n}, \nabla\hat{\chi} \right)  =  \sum_{j=1}^{n} w_j B\left( \zeta^{n-j}, \hat{\chi} \right)  +  \left(R^{n}+D_N^{\sigma}\eta^n, \hat{\chi} \right).   \label{lkx03}
\end{align}
We then prove the stability result and error estimate for the fully discrete Galerkin scheme as follows.

\begin{theorem}\label{thm5.4}  Let $U_h^n$ be the solution of the fully discrete scheme \eqref{lkx01}. The following stability result holds
 \begin{align}
      \|U^m_h\|^2 \leq Q(T) \left( \|U^0_h\|^2 + N^{-\sigma} \|\nabla U^{0}_h\|^2 +  \max\limits_{1\leq n \leq m}\|f^n\|^2 \right), \quad 1\leq m \leq N.  \label{wl01}
    \end{align}
Moreover, the following error estimate holds for $\gamma\geq \frac{1}{\sigma}$
    \begin{align}
       \|u^m-U^m\| \leq Q(T) (N^{-1}+h^2), \quad 1\leq m \leq N. \label{w1021}
    \end{align}
\end{theorem}

\begin{proof}
We follow the proof of Theorem \ref{thm4.1} and then choose $\hat{\chi}=U_h^n$ in \eqref{lkx01} to obtain \eqref{wl01}.
We take $\hat{\chi} =\zeta^n$ in \eqref{lkx03} and follow Theorem \ref{thm4.2} and \eqref{ww015} to conclude that
 \begin{align}
    \|\zeta^{\bar{m}} \|
    &  \leq  2 \sum_{n=1}^{\bar{m}} P^{(\bar{m})}_{\bar{m}-n} (\|(R_1)^n\|+\|(R_2)^n\| + \|D_N^{\sigma}\eta^n\|) \nonumber \\
     & + Q(T) \sum\limits_{n=0}^{\bar{m}-1}  \mu_{\beta,\bar{m}-n} \sum_{j=1}^{n} P^{(n)}_{n-j} (\|(R_1)^j\|+\|(R_2)^j\|+\|D_N^{\sigma}\eta^j\|).
    \label{lkx04}
\end{align}
By Theorem \ref{thm4.2}, we have
 \begin{align}
    \sum_{j=1}^{n} P^{(n)}_{n-j} (\|(R_1)^j\|+\|(R_2)^j\|) \leq Q(T)( N^{-1} + N^{-\min(2-\sigma,\gamma\sigma)} ).
    \label{lkx05}
\end{align}
Then we further estimate the term with respect to $\|D_N^{\sigma}\eta^n\|$. We first use \eqref{qwl04} to get
\begin{align}
    D_N^{\sigma}\eta^n& =  \frac{1}{\Gamma(1-\sigma)} \sum\limits_{k=1}^{n}\int_{t_{k-1}}^{t_k} \frac{(\eta^k -\eta^{k-1} )dt}{\tau_k (t_n-t)^{\sigma}}  \nonumber \\
    &= \frac{1}{\Gamma(1-\sigma)} \sum\limits_{k=1}^{n}\int_{t_{k-1}}^{t_k} \frac{1}{\tau_k} \int_{t_{k-1}}^{t_k} \frac{ (I_h - I) \partial_t u}{(t_n-t)^{\sigma}}ds dt. \label{eta}
\end{align}
For $n=1$, we use \eqref{regu}, \eqref{yq06} and $\tau_1=t_1$ to bound \eqref{eta} via
\begin{equation} \label{eta1}
\begin{split}
 \big \| D_N^{\sigma}\eta^1 \big\| &
 \le Q h^2 \int_{0}^{t_1}(t_1 -t)^{-\sigma} \frac{1}{t_1} \int_{0}^{t_1} s^{\sigma-1 }ds dt \\
&
 \le Q h^2  \int_{0}^{t_1}(t_1-t)^{-\sigma}    t_1^{\sigma - 1 } dt \le Q h^2.
\end{split}
\end{equation}
Furthermore, we estimate $D_N^{\sigma}\eta^n$ in \eqref{eta} for $2 \le n \le N$ as follows
\begin{equation} \label{etan}
\begin{split}
 \big \| D_N^{\sigma}\eta^n \big\| &
 \le Q h^2  \sum\limits_{k=1}^{n}\int_{t_{k-1}}^{t_k} (t_n-t)^{-\sigma} \frac{1}{\tau_k} \int_{t_{k-1}}^{t_k}  s^{\sigma-1 }ds dt\\
&
 = Q h^2   \int_{0}^{t_1} (t_n-t)^{-\sigma} \frac{1}{t_1}  \int_{0}^{t_1}  s^{\sigma -1 }ds dt \\
&    + Q h^2 \sum\limits_{k=2}^{n}\int_{t_{k-1}}^{t_k} (t_n-t)^{-\sigma} \frac{1}{\tau_k} \int_{t_{k-1}}^{t_k}  s^{\sigma -1 }ds dt \\
& = : \hat \Xi_1 +  \hat \Xi_2.
\end{split}
\end{equation}
We bound $\hat \Xi_1 $ in \eqref{etan} with $n\geq 2$ by
\begin{equation} \label{etan1}
\begin{split}
    \hat \Xi_1 & \le  Q h^2  \int_{0}^{t_1} (t_n-t)^{-\sigma}   t_1^{\sigma -1}  dt  \le  Q h^2  t_1^{\sigma } t_{n-1}^{-\sigma} \le  Q h^2  t_1^{\sigma } t_{1}^{-\sigma} = Q h^2.
\end{split}
\end{equation}
We further estimate $\hat \Xi_2 $ in \eqref{etan} with $s^{\sigma-1 }\leq Q t^{\sigma-1 }$ ($s,t\in [t_{k-1},t_k]$, $k\geq 2$) by
 \begin{equation} \label{etan2}
 \begin{split}
      \hat \Xi_2 & \le Q h^2  \sum\limits_{k=2}^{n}\int_{t_{k-1}}^{t_k} (t_n-t)^{-\sigma} t^{\sigma-1 }dt \\
      & \qquad = Q h^2  \int_{t_1}^{t_n} (t_n-t)^{-\sigma} t^{\sigma-1 }dt  \\
      &  \qquad \le Q h^2  \int_{0}^{t_n} (t_n-t)^{-\sigma} t^{\sigma-1 }dt  \le Q h^2.
 \end{split}
\end{equation}
We combine \eqref{eta}--\eqref{etan2} to conclude that $\|D_N^{\sigma}\eta^n\| \le Qh^2$ for $1 \le n \le m \le N$, which, together with \eqref{lkx04}--\eqref{lkx05} and Lemma \ref{lem03}, gives
 \begin{equation*}
    \begin{split}
         \|\zeta^{m}\| \le \|\zeta^{\bar{m}}\|  \ds \le Q (N^{-1} +h^2).
    \end{split}
 \end{equation*}
 We invoke this with \eqref{yq06} to obtain \eqref{w1021} and thus complete the proof.
\end{proof}

\section{Numerical experiment}\label{sec4}
We carry out numerical experiments to investigate the convergence behavior of the fully discrete Galerkin scheme \eqref{lkx01}. Following the theoretical analysis, we always set $\gamma=1/\sigma$. The uniform spatial partition is adopted for both one-dimensional and two-dimensional domains.

\vskip 1mm
\textbf{Example 1: Positive-type power-law kernel}
Let  $\Omega = (0,1)$, $T=1$, $\mu_1= \mu_2 = 1$,  $K(t)=t^{-\beta}$, $u_0(x)=\sin(\pi x)$, and $f=1$.
To test the convergence behavior of the spatial $L^2$ error at $t=T$, we denote the following temporal error and the convergence rate by the two-mesh principle (see \cite[p.~107]{Farrell})
\begin{align*}
  & E_\tau(N,M)= \sqrt{h\sum_{j=1}^{M-1} \left(U^{2{N}}_j-U^N_j\right)^2},\quad  \text{Rate}_\tau = \log_{2}\left(\frac{E_\tau(N,M)}{E_\tau(2N,M)}\right),
  \end{align*}
  and similarly define the spatial error and the corresponding convergence rate
  \begin{align*}
  E_h(N,M)=   \sqrt{h\sum_{j=1}^{M-1}\left(U_{2j}^{N}-U_j^N\right)^2}, \quad \text{Rate}_h = \log_{2}\left(\frac{E_h(N,M)}{E_h(N,2M)}\right).
\end{align*}
Numerical results are presented in Tables \ref{tab10}--\ref{tab11} under different parameters $\beta$ and $\sigma$, which indicates the first-order accuracy in time and second-order accuracy in space that are consistent with the theoretical analysis.

\begin{table}
    \center \footnotesize
    \caption{Errors and temporal convergence rates for Example 1 under $M=32$.}
    \label{tab10}
    \begin{tabular}{cccccccccccc}
      \toprule
    & & & \multicolumn{2}{c}{$\beta=0.1$} & &\multicolumn{2}{c}{$\beta=0.2$}\\
   \cmidrule{4-5}  \cmidrule{7-8}
      $\sigma=0.2$ &$N$ & & {$E_{\tau}$} & {$\text{Rate}_{\tau}$}  & $N$ & {$E_{\tau}$} & {$\text{Rate}_{\tau}$} \\
      \midrule
       & $128$    & &  $7.1575 \times 10^{-3}$  &    *         & $128$    & $7.4800 \times 10^{-3}$  &    *    \\
       & $256$   & &  $3.4648 \times 10^{-3}$  &    1.05      & $256$   & $3.6153 \times 10^{-3}$  & 1.05   \\
       & $512$   & &  $1.6736 \times 10^{-3}$  &    1.05      & $512$   & $1.7462 \times 10^{-3}$  & 1.05    \\
       & $1024$   & &  $7.8534 \times 10^{-4}$  &    1.09      & $1024$   & $8.2099 \times 10^{-4}$  & 1.09    \\
       \midrule
        & & & \multicolumn{2}{c}{$\beta=0.3$} & &\multicolumn{2}{c}{$\beta=0.6$}\\
   \cmidrule{4-5}  \cmidrule{7-8}
      $\sigma=0.4$&$N$ & & {$E_{\tau}$} & {$\text{Rate}_{\tau}$}  & $N$ & {$E_{\tau}$} & {$\text{Rate}_{\tau}$}\\
      \midrule
       & $512$    & &  $1.0267 \times 10^{-3}$  &    *         & $512$    & $3.5532 \times 10^{-4}$  &    *    \\
       & $1024$   & &  $4.9380 \times 10^{-4}$  &    1.06      & $1024$   & $1.7163 \times 10^{-4}$  & 1.05   \\
       & $2048$   & &  $2.3917 \times 10^{-4}$  &    1.05      & $2048$   & $8.3450 \times 10^{-5}$  & 1.04    \\
       & $4096$   & &  $1.1666 \times 10^{-4}$  &    1.04      & $4096$   & $4.0836 \times 10^{-5}$  & 1.03    \\
        \midrule
        & & & \multicolumn{2}{c}{$\beta=0.5$} & &\multicolumn{2}{c}{$\beta=0.8$}\\
   \cmidrule{4-5}  \cmidrule{7-8}
      $\sigma=0.7$ &$N$ & & {$E_{\tau}$} & {$\text{Rate}_{\tau}$}  & $N$ & {$E_{\tau}$} & {$\text{Rate}_{\tau}$} \\
      \midrule
       & $512$   & &  $2.5981 \times 10^{-4}$  &    *       & $512$   & $6.4781 \times 10^{-5}$  &   *    \\
       & $1024$  & &  $1.2536 \times 10^{-4}$  &    1.05    & $1024$  & $3.1617 \times 10^{-5}$  & 1.03   \\
       & $2048$  & &  $6.0415 \times 10^{-5}$  &    1.05    & $2048$  & $1.5328 \times 10^{-5}$  & 1.04    \\
       & $4096$  & &  $2.9122 \times 10^{-5}$  &    1.05    & $4096$  & $7.4075 \times 10^{-6}$  & 1.05    \\
      \bottomrule
    \end{tabular}
\end{table}

\begin{table}
    \center \footnotesize
    \caption{Errors and spatial convergence rates for Example 1 under $N=256$.}
    \label{tab11}
    \begin{tabular}{cccccccccccc}
      \toprule
    & & & \multicolumn{2}{c}{$\beta=0.1$} & &\multicolumn{2}{c}{$\beta=0.2$}\\
   \cmidrule{4-5}  \cmidrule{7-8}
      $\sigma=0.2$ &$M$ & & {$E_{h}$} & {$\text{Rate}_{h}$}  & $M$ & {$E_{h}$} & {$\text{Rate}_{h}$} \\
      \midrule
       & $64$    & &  $5.0188 \times 10^{-5}$  &    *         & $64$    & $3.6960 \times 10^{-5}$  &    *    \\
       & $128$   & &  $1.2567 \times 10^{-5}$  &    2.00      & $128$   & $9.2810 \times 10^{-6}$  & 2.00   \\
       & $256$   & &  $3.1433 \times 10^{-6}$  &    2.00      & $256$   & $2.3249 \times 10^{-6}$  & 2.00    \\
       & $512$   & &  $7.8595 \times 10^{-7}$  &    2.00      & $512$   & $5.8177 \times 10^{-7}$  & 2.00    \\
       \midrule
        & & & \multicolumn{2}{c}{$\beta=0.3$} & &\multicolumn{2}{c}{$\beta=0.6$}\\
   \cmidrule{4-5}  \cmidrule{7-8}
      $\sigma=0.4$&$M$ & & {$E_{h}$} & {$\text{Rate}_{h}$}  & $M$ & {$E_{h}$} & {$\text{Rate}_{h}$}\\
      \midrule
       & $64$    & &  $1.2542 \times 10^{-5}$  &    *         & $64$    & $4.5976 \times 10^{-6}$  &    *    \\
       & $128$   & &  $3.1470 \times 10^{-6}$  &    1.99      & $128$   & $1.1916 \times 10^{-6}$  & 1.95   \\
       & $256$   & &  $7.8761 \times 10^{-7}$  &    2.00      & $256$   & $3.0302 \times 10^{-7}$  & 1.98    \\
       & $512$   & &  $1.9698 \times 10^{-7}$  &    2.00      & $512$   & $7.6384 \times 10^{-8}$  & 1.99    \\
        \midrule
        & & & \multicolumn{2}{c}{$\beta=0.5$} & &\multicolumn{2}{c}{$\beta=0.8$}\\
   \cmidrule{4-5}  \cmidrule{7-8}
      $\sigma=0.7$ &$M$ & & {$E_{h}$} & {$\text{Rate}_{h}$}  & $M$ & {$E_{h}$} & {$\text{Rate}_{h}$} \\
      \midrule
       & $64$    & &  $2.6147 \times 10^{-5}$  &    *         & $64$    & $2.0143 \times 10^{-5}$  &    *    \\
       & $128$   & &  $6.5977 \times 10^{-6}$  &    1.99      & $128$   & $5.0702 \times 10^{-6}$  & 1.99   \\
       & $256$   & &  $1.6569 \times 10^{-6}$  &    1.99      & $256$   & $1.2718 \times 10^{-6}$  & 2.00    \\
       & $512$   & &  $4.1516 \times 10^{-7}$  &    2.00      & $512$   & $3.1846 \times 10^{-7}$  & 2.00    \\
      \bottomrule
    \end{tabular}
\end{table}

\vskip 1mm
\textbf{Example 2: Positive-type tempered power-law kernel}
Let  $\Omega = (0,1)$, $T=1$, $\mu_1= \mu_2 = 1$, $K(t)=e^{-t}t^{-\beta}$, $u_0(x)=\sin(\pi x)$, and the source term $f=0$. The errors and convergence rates are measured as above and presented in Tables \ref{tab12}--\ref{tab13}, which again demonstrate the first-order accuracy in time and second-order accuracy in space.

\begin{table}
    \center \footnotesize
    \caption{Errors and temporal convergence rates for Example 2 under $M=32$.}
    \label{tab12}
    \begin{tabular}{cccccccccccc}
      \toprule
    & & & \multicolumn{2}{c}{$\beta=0.1$} & &\multicolumn{2}{c}{$\beta=0.2$}\\
   \cmidrule{4-5}  \cmidrule{7-8}
      $\sigma=0.2$ &$N$ & & {$E_{\tau}$} & {$\text{Rate}_{\tau}$}  & $N$ & {$E_{\tau}$} & {$\text{Rate}_{\tau}$} \\
      \midrule
       & $64$    & &  $8.6267 \times 10^{-3}$  &    *         & $64$    & $9.1278 \times 10^{-3}$  &    *    \\
       & $128$   & &  $4.1032 \times 10^{-3}$  &    1.07      & $128$   & $4.3227 \times 10^{-3}$  & 1.08   \\
       & $256$   & &  $1.9616 \times 10^{-3}$  &    1.06      & $256$   & $2.0640 \times 10^{-3}$  & 1.07    \\
       & $512$   & &  $9.4456 \times 10^{-4}$  &    1.05      & $512$   & $9.9413 \times 10^{-4}$  & 1.05    \\
       \midrule
        & & & \multicolumn{2}{c}{$\beta=0.3$} & &\multicolumn{2}{c}{$\beta=0.6$}\\
   \cmidrule{4-5}  \cmidrule{7-8}
      $\sigma=0.4$&$N$ & & {$E_{\tau}$} & {$\text{Rate}_{\tau}$}  & $N$ & {$E_{\tau}$} & {$\text{Rate}_{\tau}$}\\
      \midrule
       & $64$   & &  $4.1793 \times 10^{-3}$  &    *       & $64$   & $4.6180 \times 10^{-4}$  &   *    \\
       & $128$  & &  $1.9159 \times 10^{-3}$  &    1.13    & $128$  & $2.1188 \times 10^{-4}$  & 1.12   \\
       & $256$  & &  $8.9827 \times 10^{-4}$  &    1.09    & $256$  & $9.9581 \times 10^{-5}$  & 1.09    \\
       & $512$  & &  $4.2750 \times 10^{-4}$  &    1.07    & $512$  & $4.7262 \times 10^{-5}$  & 1.08    \\
              \midrule
        & & & \multicolumn{2}{c}{$\beta=0.5$} & &\multicolumn{2}{c}{$\beta=0.8$}\\
   \cmidrule{4-5}  \cmidrule{7-8}
      $\sigma=0.7$ &$N$ & & {$E_{\tau}$} & {$\text{Rate}_{\tau}$}  & $N$ & {$E_{\tau}$} & {$\text{Rate}_{\tau}$} \\
      \midrule
       & $64$   & &  $5.9709 \times 10^{-4}$  &    *       & $64$   & $1.3716 \times 10^{-4}$  &   *    \\
       & $128$  & &  $2.9394 \times 10^{-4}$  &    1.02    & $128$  & $7.2378 \times 10^{-5}$  & 0.92   \\
       & $256$  & &  $1.4206 \times 10^{-4}$  &    1.05    & $256$  & $3.6660 \times 10^{-5}$  & 0.98    \\
       & $512$  & &  $6.7966 \times 10^{-5}$  &    1.06    & $512$  & $1.8113 \times 10^{-5}$  & 1.02    \\
      \bottomrule
    \end{tabular}
\end{table}

\begin{table}
    \center \footnotesize
    \caption{Errors and spatial convergence rates for Example 2 under $N=128$.}
    \label{tab13}
    \begin{tabular}{cccccccccccc}
      \toprule
    & & & \multicolumn{2}{c}{$\beta=0.1$} & &\multicolumn{2}{c}{$\beta=0.2$}\\
   \cmidrule{4-5}  \cmidrule{7-8}
      $\sigma=0.2$ &$M$ & & {$E_{h}$} & {$\text{Rate}_{h}$}  & $M$ & {$E_{h}$} & {$\text{Rate}_{h}$} \\
      \midrule
       & $64$    & &  $3.1192 \times 10^{-6}$  &    *         & $64$    & $2.0487 \times 10^{-6}$  &    *    \\
       & $128$   & &  $7.8466 \times 10^{-7}$  &    1.99      & $128$   & $5.0743 \times 10^{-7}$  & 2.01   \\
       & $256$   & &  $1.9647 \times 10^{-7}$  &    2.00      & $256$   & $1.2656 \times 10^{-7}$  & 2.00    \\
       & $512$   & &  $4.9137 \times 10^{-8}$  &    2.00      & $512$   & $3.1621 \times 10^{-8}$  & 2.00    \\
       \midrule
        & & & \multicolumn{2}{c}{$\beta=0.3$} & &\multicolumn{2}{c}{$\beta=0.6$}\\
   \cmidrule{4-5}  \cmidrule{7-8}
      $\sigma=0.4$&$M$ & & {$E_{h}$} & {$\text{Rate}_{h}$}  & $M$ & {$E_{h}$} & {$\text{Rate}_{h}$}\\
      \midrule
       & $64$    & &  $7.3448 \times 10^{-6}$  &    *         & $64$    & $3.4129 \times 10^{-6}$  &    *    \\
       & $128$   & &  $1.8350 \times 10^{-6}$  &    2.00      & $128$   & $8.5378 \times 10^{-7}$  & 2.00   \\
       & $256$   & &  $4.5866 \times 10^{-7}$  &    2.00      & $256$   & $2.1348 \times 10^{-7}$  & 2.00    \\
       & $512$   & &  $1.1466 \times 10^{-7}$  &    2.00      & $512$   & $5.3372 \times 10^{-8}$  & 2.00    \\
        \midrule
        & & & \multicolumn{2}{c}{$\beta=0.5$} & &\multicolumn{2}{c}{$\beta=0.8$}\\
   \cmidrule{4-5}  \cmidrule{7-8}
      $\sigma=0.7$ &$M$ & & {$E_{h}$} & {$\text{Rate}_{h}$}  & $M$ & {$E_{h}$} & {$\text{Rate}_{h}$} \\
      \midrule
       & $64$    & &  $6.1362 \times 10^{-7}$  &    *         & $64$    & $1.6101 \times 10^{-6}$  &    *    \\
       & $128$   & &  $1.5351 \times 10^{-7}$  &    2.00      & $128$   & $4.0274 \times 10^{-7}$  & 2.00   \\
       & $256$   & &  $3.8384 \times 10^{-8}$  &    2.00      & $256$   & $1.0070 \times 10^{-7}$  & 2.00    \\
       & $512$   & &  $9.5963 \times 10^{-9}$  &    2.00      & $512$   & $2.5176 \times 10^{-8}$  & 2.00    \\
      \bottomrule
    \end{tabular}
\end{table}

\vskip 1mm
\textbf{Example 3: Non-positive-type power-law kernel}
Let $\Omega = (0,1)\times (0,1)$, $T=1$, $\mu_1=1$, $\mu_2 = 0$,  $K(t)=-t^{-\beta}$, $u_0\equiv 0$, and the source term $f=1$.
In the two-dimensional case, we define the following $L^2$ error and the convergence rates
\begin{align*}
  & E^\tau(N,M)= \sqrt{h_x\sum_{i=1}^{M-1}h_y\sum_{j=1}^{M-1} \left(U^{2{N}}_{i,j}-U^N_{i,j}\right)^2},\quad  \text{Rate}^\tau = \log_{2}\left(\frac{E^\tau(N,M)}{E^\tau(2N,M)}\right), \\
  & E^h(N,M)=   \sqrt{h_x\sum_{i=1}^{M-1}h_y\sum_{j=1}^{M-1}\left(U_{2i,2j}^{N}-U_{i,j}^N\right)^2}, \quad \text{Rate}^h = \log_{2}\left(\frac{E^h(N,M)}{E^h(N,2M)}\right).
\end{align*}
The errors and convergence rates are measured as above and presented in Tables \ref{tab14}--\ref{tab15}, which demonstrate the first-order accuracy in time and second-order accuracy in space for the non-positive kernel case.

\begin{table}
    \center \footnotesize
    \caption{Errors and temporal convergence rates for Example 3 under $M=16$.}
    \label{tab14}
    \vskip 2mm
    \begin{tabular}{cccccccccccc}
      \toprule
    & & & \multicolumn{2}{c}{$\beta=0.1$} & &\multicolumn{2}{c}{$\beta=0.2$}\\
   \cmidrule{4-5}  \cmidrule{7-8}
      $\sigma=0.2$ &$N$ & & {$E^{\tau}$} & {$\text{Rate}^{\tau}$}  & $N$ & {$E^{\tau}$} & {$\text{Rate}^{\tau}$} \\
      \midrule
       & $64$    & &  $1.6558 \times 10^{-3}$  &    *         & $64$    & $1.7542 \times 10^{-3}$  &    *    \\
       & $128$   & &  $8.0560 \times 10^{-4}$  &    1.04      & $128$   & $8.5165 \times 10^{-4}$  & 1.04   \\
       & $256$   & &  $3.8905 \times 10^{-4}$  &    1.05      & $256$   & $4.1107 \times 10^{-4}$  & 1.05    \\
       & $512$   & &  $1.8741 \times 10^{-4}$  &    1.05      & $512$   & $1.9807 \times 10^{-4}$  & 1.05    \\
       \midrule
        & & & \multicolumn{2}{c}{$\beta=0.25$} & &\multicolumn{2}{c}{$\beta=0.55$}\\
   \cmidrule{4-5}  \cmidrule{7-8}
      $\sigma=0.4$&$N$ & & {$E^{\tau}$} & {$\text{Rate}^{\tau}$}  & $N$ & {$E^{\tau}$} & {$\text{Rate}^{\tau}$}\\
      \midrule
       & $64$   & &  $2.0051 \times 10^{-3}$  &    *       & $64$   & $1.0417 \times 10^{-2}$  &   *    \\
       & $128$  & &  $9.9223 \times 10^{-4}$  &    1.02    & $128$  & $5.3782 \times 10^{-3}$  & 0.95   \\
       & $256$  & &  $4.8326 \times 10^{-4}$  &    1.04    & $256$  & $2.7000 \times 10^{-3}$  & 0.94    \\
       & $512$  & &  $2.3340 \times 10^{-4}$  &    1.05    & $512$  & $1.3362 \times 10^{-3}$  & 1.01    \\
              \midrule
        & & & \multicolumn{2}{c}{$\beta=0.3$} & &\multicolumn{2}{c}{$\beta=0.6$}\\
   \cmidrule{4-5}  \cmidrule{7-8}
      $\sigma=0.6$ &$N$ & & {$E^{\tau}$} & {$\text{Rate}^{\tau}$}  & $N$ & {$E^{\tau}$} & {$\text{Rate}^{\tau}$} \\
      \midrule
       & $64$   & &  $3.5014 \times 10^{-3}$  &    *       & $64$   & $9.6638 \times 10^{-2}$  &   *    \\
       & $128$  & &  $1.7820 \times 10^{-3}$  &    0.97    & $128$  & $5.4491 \times 10^{-2}$  & 0.83   \\
       & $256$  & &  $8.8840 \times 10^{-4}$  &    1.00    & $256$  & $2.8725 \times 10^{-2}$  & 0.92    \\
       & $512$  & &  $4.3725 \times 10^{-4}$  &    1.02    & $512$  & $1.4630 \times 10^{-2}$  & 0.97    \\
      \bottomrule
    \end{tabular}
\end{table}

\begin{table}
    \center \footnotesize
    \caption{Errors and spatial convergence rates for Example 3 under $N=32$.}
    \label{tab15}
    \vskip 2mm
    \begin{tabular}{cccccccccccc}
      \toprule
    & & & \multicolumn{2}{c}{$\beta=0.1$} & &\multicolumn{2}{c}{$\beta=0.2$}\\
   \cmidrule{4-5}  \cmidrule{7-8}
      $\sigma=0.2$ &$M$ & & {$E^{h}$} & {$\text{Rate}^{h}$}  & $M$ & {$E^{h}$} & {$\text{Rate}^{h}$} \\
      \midrule
       & $8$    & &  $1.9866 \times 10^{-3}$  &    *         & $8$    & $2.1645 \times 10^{-3}$  &    *    \\
       & $16$   & &  $4.7485 \times 10^{-4}$  &    2.06      & $16$   & $5.1805 \times 10^{-4}$  & 2.06   \\
       & $32$   & &  $1.1636 \times 10^{-4}$  &    2.03      & $32$   & $1.2701 \times 10^{-4}$  & 2.03    \\
       & $64$   & &  $2.8910 \times 10^{-5}$  &    2.01      & $64$   & $3.1560 \times 10^{-5}$  & 2.01    \\
       \midrule
        & & & \multicolumn{2}{c}{$\beta=0.25$} & &\multicolumn{2}{c}{$\beta=0.55$}\\
   \cmidrule{4-5}  \cmidrule{7-8}
      $\sigma=0.4$&$M$ & & {$E^{h}$} & {$\text{Rate}^{h}$}  & $M$ & {$E^{h}$} & {$\text{Rate}^{h}$}\\
      \midrule
       & $8$    & &  $2.5944 \times 10^{-3}$  &    *         & $8$    & $8.4393 \times 10^{-3}$  &    *    \\
       & $16$   & &  $6.2309 \times 10^{-4}$  &    2.06      & $16$   & $2.0426 \times 10^{-3}$  & 2.05   \\
       & $32$   & &  $1.5302 \times 10^{-4}$  &    2.03      & $32$   & $5.0262 \times 10^{-4}$  & 2.02    \\
       & $64$   & &  $3.8043 \times 10^{-5}$  &    2.01      & $64$   & $1.2504 \times 10^{-4}$  & 2.01    \\
        \midrule
        & & & \multicolumn{2}{c}{$\beta=0.3$} & &\multicolumn{2}{c}{$\beta=0.6$}\\
   \cmidrule{4-5}  \cmidrule{7-8}
      $\sigma=0.6$ &$M$ & & {$E^{h}$} & {$\text{Rate}^{h}$}  & $M$ & {$E^{h}$} & {$\text{Rate}^{h}$} \\
      \midrule
       & $8$    & &  $3.7462 \times 10^{-3}$  &    *         & $8$    & $4.8490 \times 10^{-2}$  &    *    \\
       & $16$   & &  $9.0422 \times 10^{-4}$  &    2.05      & $16$   & $1.1682 \times 10^{-2}$  & 2.05   \\
       & $32$   & &  $2.2245 \times 10^{-4}$  &    2.02      & $32$   & $2.8645 \times 10^{-3}$  & 2.03    \\
       & $64$   & &  $5.5336 \times 10^{-5}$  &    2.01      & $64$   & $7.1182 \times 10^{-4}$  & 2.01    \\
      \bottomrule
    \end{tabular}
\end{table}

\section{Conclusion}\label{sec5}
This work considers the subdiffusion problem with non-positive memory,  which could be transformed from subdiffusion or subdiffusive Fokker-Planck equation with variable exponent.  We apply the non-uniform L1 formula and interpolation quadrature to discretize the fractional derivative and the memory term, respectively, and then adopt the complementary discrete convolution kernel approach to prove the stability and first-order accuracy in time of the scheme. Numerical experiments are performed to substantiate the theoretical results. The developed ideas and methods could also be extended to investigate numerical discretization of other evolution equations with non-positive memory.

\vskip 0.15in
{\footnotesize\noindent \textbf{Acknowledgments}
 This work was partially supported by the National Key R\&D Program of China (No. 2023YFA1008903), the National Natural Science Foundation of China (No. 12301555), the Taishan Scholars Program of Shandong Province (No. tsqn202306083), the Postdoctoral Fellowship Program of CPSF (No. GZC20240938).
}

\vskip 0.15in
{\footnotesize\noindent \textbf{Data availability} No data is used in this work.}

\section*{Declarations}

{\footnotesize\noindent \textbf{Conflict of interest} The authors declare that they have no conflict of interest.}

\end{document}